\documentclass[12pt,twoside,a4paper]{amsart}

\usepackage{amssymb,amsmath,amsthm}

\usepackage{color}

\newtheorem{theorem}{Theorem}[section]
\newtheorem{thm}[theorem]{Theorem}

\newtheorem{coro}[theorem]{Corollary}

\theoremstyle{plain}
\newtheorem*{namedthm}{\namedthmname}
\newcounter{namedthm}



\newcommand{\R}{\mathbb{R}}
\newcommand{\C}{\mathbb{C}}
\newcommand{\N}{\mathbb{N}}

\numberwithin{equation}{section}

\usepackage{hyperref}
\hypersetup{
	bookmarks=true,         
	unicode=false,         
	pdftoolbar=true,       
	pdfmenubar=true,       
	pdffitwindow=false,    
	pdfstartview={FitH},   
	pdftitle={Pluripotential vs Viscosity solutions to complex Monge-Amp\`ere flows},    
	pdfauthor={Vincent Guedj, Chinh H. Lu, and Ahmed Zeriahi},     
	colorlinks=true,       
	linkcolor=black,         
	citecolor=black,        
	filecolor=black,      
	urlcolor=black}          

\begin{document}
	\def\K{\mathbb{K}}
	\def\R{\mathbb{R}}
	\def\C{\mathbb{C}}
	\def\Z{\mathbb{Z}}
	\def\Q{\mathbb{Q}}
	\def\D{\mathbb{D}}
	\def\N{\mathbb{N}}
	\def\T{\mathbb{T}}
	\def\P{\mathbb{P}}
	\def\A{\mathscr{A}}
	\def\CC{\mathscr{C}}
	\renewcommand{\theequation}{\thesection.\arabic{equation}}
	\renewenvironment{proof}{{\bfseries Proof:}}{\qed}
	\renewcommand{\thelemme}{\empty{}}
	\newtheorem{cond}{C}
	\newtheorem{lemma}{Lemma}[section]
	\newtheorem{corollary}{Corollary}[section]
	\newtheorem{proposition}{Proposition}[section]
	\newtheorem{notation}{Notation}[section]
	\newtheorem{remark}{Remark}[section]
	\newtheorem{example}{Example}[section]
	\newtheorem{probleme}{Problem}[section]
	\bibliographystyle{plain}

\title[Study of the operator $\partial^k \bar{\partial}^{k} + c$~~]{\textbf {Study of the operator} $\partial^{k} \bar{\partial}^{k} + c$ \textbf {in the weighted Hilbert space} $L^2(\mathbb{\C}, \mathrm{\textnormal{e}}^{\mathrm{-\vert z \vert^2}})$.}
\author[ E.\  Bodian    \& S.\  Sambou \& P.\  Badiane \&  W.\ O.\  Ingoba  \& S.\  Sambou ]{ Eramane Bodian  \& Souhaibou Sambou \& Papa Badiane \& Winnie Ossete Ingoba \& Salomon Sambou }
\address{ Mathematics Department\\UFR of Science and Technology \\  Assane Seck University of Ziguinchor, BP: 523 (Senegal)}
\email{m.bodian@univ-zig.sn}
\address{Mathematics Department\\UFR of Science and Technology \\  Assane Seck University of Ziguinchor, BP: 523 (Senegal)}
\email{s.sambou1440@zig.univ.sn }
\address{Mathematics Department\\UFR of Science and Technology \\  Assane Seck University of Ziguinchor, BP: 523 (Senegal)}
\email{p.badiane4963@zig.univ.sn}

\address{Mathematics Department\\ Faculty of Science and Technology \\  Marien Ngouabi University, (Congo)}
\email{wnnossete@gmail.com}

\address{Mathematics Department\\UFR of Science and Technology \\  Assane Seck University of Ziguinchor, BP: 523 (Senegal)}
\email{ssambou@univ-zig.sn }
 


\maketitle
\renewcommand{\abstractname}{Abstract}
\begin{abstract}
By the H\"ormander's $L^2$-method, we study the operator $\partial^k \bar{\partial}^{k} + c$ for any order $k$ in the weighted Hilbert space $L^2(\mathbb{\C}, \mathrm{\textnormal{e}}^{-|z|^2})$. We prove the existence of its right inverse witch is also a bounded operator.  

\vskip 2mm
\noindent
\end{abstract}  
\renewcommand{\abstractname}{Abstract}
\section{Introduction}
In \cite{1}, Shoayu DAI and Yifei PAN are interested to the operator's study $\frac{d^k}{dx^k} + a$. This work has been extended in \cite{2} to the operator $\bar{\partial}^k + a$. Inspired by the $\partial \bar{\partial}$ problem, we study the operator $\partial^k \bar{\partial}^k + c$ in the weighted Hilbert space $L^2(\mathbb{\C}, \mathrm{\textnormal{e}}^{-|z|^2})$.\\
We prove the existence of a weak solution of the equation $$\partial^k \bar{\partial}^k u + cu = f$$ where $u, f \in L^2(\mathbb{\C}, \mathrm{\textnormal{e}}^{- \varphi})$, $c$ is a complex constant, $\bar{\partial}^{k} = \frac{\partial^k}{\partial \bar{z}^k}$ and $\partial^{k} = \frac{\partial^k}{\partial z^k}$. We thus have the following result:
\begin{thm} \label{P}
For any $f \in L^2(\mathbb{\C}, \mathrm{\textnormal{e}}^{- \vert z \vert^2})$, there exists a weak solution $u \in L^2(\mathbb{\C}, \mathrm{\textnormal{e}}^{- \vert z \vert^2})$ of the equation
$$ \partial^k \bar{\partial}^{k} u + cu = f$$ with the norm estimate  
$$\int_\mathbb{C} \vert u \vert^2 \mathrm{\textnormal{e}}^{- \vert z \vert^2} d\sigma \leq \frac{1}{(k!)^2} \int_\mathbb{C} \vert f \vert^2 \mathrm{\textnormal{e}}^{- \vert z \vert^2} d\sigma.$$
Moreover for $k = 1$, we have for any smooth and nonnegative function $\varphi$ define on $\mathbb{C}$ such that $\Delta(\mathrm{\textnormal{e}}^{\varphi} \Delta \mathrm{\textnormal{e}}^{- \varphi })> 0$. For all $ f \in L^2(\mathbb{\C}, \mathrm{\textnormal{e}}^{- \varphi})$ such that $$\frac{f}{\sqrt{\Delta(\mathrm{\textnormal{e}}^{\varphi} \Delta \mathrm{\textnormal{e}}^{-\varphi})}}  \in L^2(\mathbb{\C}, \mathrm{\textnormal{e}}^{-\varphi}),$$ there exists a weak solution $ u \in  L^2(\mathbb{\C}, \mathrm{\textnormal{e}}^{-\varphi})$ of the equation 
$$ \partial^1 \bar{\partial}^1 u + cu = f$$ with the norm estimate   
$$\int_\mathbb{C} \vert u \vert^2 \mathrm{\textnormal{e}}^{- \varphi} d\sigma \leq 16 \int_\mathbb{C} \frac{\vert f \vert^2}{\Delta(\mathrm{\textnormal{e}}^{ \varphi} \Delta \mathrm{\textnormal{e}}^{- \varphi})}  \mathrm{\textnormal{e}}^{-\varphi} d\sigma.$$
\end{thm}
As a consequence of theorem \ref{P}, we show that the operator $\partial^k \bar{\partial}^k + c$ has a right inverse and that its inverse is bounded in  $L^2(\mathbb{\C}, \mathrm{\textnormal{e}}^{- \vert z \vert^2})$.\\
For $k = 1$ then $4  \partial^1 \bar{\partial}^1 = \Delta $ is the complex Laplacian and $\Delta$ is the Laplacian in $\mathbb{R}^2$.
\section{Some Lemmas}
Consider the weighted Hilbert space 

$$ L^2(\mathbb{\C}, \mathrm{\textnormal{e}}^{- \varphi}) = \{ f \in L_{loc}^2(\mathbb{\C}) : \int_{\mathbb{C}} \vert f \vert^2 \mathrm{\textnormal{e}}^{- \varphi } d\sigma < + \infty \}$$
where $\varphi$ is a nonnegative function on $\mathbb{C}$.\\
We define the weighted inner product on $L^2(\mathbb{\C}, \mathrm{\textnormal{e}}^{- \varphi})$ by: $\forall$ $f, g \in L^2(\mathbb{\C}, \mathrm{\textnormal{e}}^{- \varphi})$, we have
$$<f, g>_\varphi = \int_\mathbb{C} \overline{f}g \mathrm{\textnormal{e}}^{- \varphi} d\sigma$$
and the norm on $L^2(\mathbb{\C}, \mathrm{\textnormal{e}}^{- \varphi})$ by
$$\parallel f \parallel_\varphi = \sqrt{<f, f>_\varphi} \; \; \; \; \; \; \; \; \forall \; \; \; \; f \in L^2(\mathbb{\C}, \mathrm{\textnormal{e}}^{- \varphi}).$$
We denote by $C_0^\infty(\mathbb{C})$ the space of the functions $\phi : \mathbb{C} \rightarrow \mathbb{C}$ of class $C^\infty$ with compact support.\\
Let $u,  f \in L_{loc}^2(\mathbb{\C})$, we say that $u$ is a weak solution of the equation $\partial^k \bar{\partial}^{k} u = f$ if
$$ \int_\mathbb{C} u \partial^k \bar{\partial}^{k} \phi d\sigma = \int_\mathbb{C} f \phi d\sigma .$$
Let $\varphi$ be a nonnegative function of class $C^\infty$ on $\mathbb{C}$ and $\phi \in C_0^\infty(\mathbb{C})$, we define the formal adjoint of $\partial^k \bar{\partial}^{k}$ with respect to the definition of weighted inner product in  $L^2(\mathbb{\C}, \mathrm{\textnormal{e}}^{- \varphi})$ and with respect to the definition of the weak solution as :
\begin{eqnarray*}
<\phi, \partial^k \bar{\partial}^{k} u> & = & \int_{\mathbb{C}} \overline{\phi} \partial^k \bar{\partial}^{k} u \mathrm{\textnormal{e}}^{- \varphi} d\sigma\\
& = &  \int_{\mathbb{C}}  \bar{\partial}^{k} \partial^k( \overline{\phi}\mathrm{\textnormal{e}}^{- \varphi})  u  d\sigma\\
& = &  \int_{\mathbb{C}} \mathrm{\textnormal{e}}^{ \varphi} \bar{\partial}^{k} \partial^k ( \overline{\phi}\mathrm{\textnormal{e}}^{- \varphi})  u \mathrm{\textnormal{e}}^{- \varphi} d\sigma\\
& = & \int_{\mathbb{C}} \overline{\mathrm{\textnormal{e}}^{ \varphi} \partial^{k} \bar{\partial}^k ( \phi \mathrm{\textnormal{e}}^{- \varphi})} u \mathrm{\textnormal{e}}^{- \varphi} d\sigma\\
& = & <\mathrm{\textnormal{e}}^{ \varphi} \partial^{k} \bar{\partial}^k ( \phi \mathrm{\textnormal{e}}^{- \varphi}), u >_\varphi\\
& =: & < \bar{\partial}_\varphi^{k*} \partial_\varphi^{k*}( \phi), u >_\varphi
\end{eqnarray*}
where $ \bar{\partial}_\varphi^{k*} \partial_\varphi^{k*}( \phi) = \mathrm{\textnormal{e}}^{ \varphi} \partial^{k} \bar{\partial}^k ( \phi \mathrm{\textnormal{e}}^{- \varphi})$ is the formal adjoint of  $\partial^k \bar{\partial}^{k}$ in $C_0^\infty(\mathbb{C})$.\\
Let $(\partial^k \bar{\partial}^{k} + c)_\varphi^*$ be the formal adjoint of $\partial^k \bar{\partial}^{k} + c$ in $C_0^\infty(\mathbb{C})$.\\
Let $I_\varphi^* = I$ where $I$ the identity operator, then $$ (\partial^k \bar{\partial}^{k} + c)_\varphi^* = \bar{\partial}_\varphi^{k*} \partial_\varphi^{k*} + c.$$
We start by giving some lemmas in the general case of
weighted spaces based on the functional analysis which will be very useful to us in the proof of the Theorem \ref{P}.
\begin{lemma} \label{A}
For each $f \in L^2(\mathbb{\C}, \mathrm{\textnormal{e}}^{- \varphi})$, there exists a weak solution $u \in L^2(\mathbb{\C}, \mathrm{\textnormal{e}}^{- \varphi})$ of the equation
$$ \partial^k \bar{\partial}^{k} u + cu = f$$ with the norm estimate  
$$\vert \vert u \vert \vert_\varphi^2 \leq a$$ if and only if
$$\vert <f, \phi>_\varphi \vert^2 \leq a \vert \vert (\partial^k \bar{\partial}^{k} + c)_\varphi^* \phi \vert \vert_\varphi ^2$$
$\forall$ $\phi \in C_0^\infty(\mathbb{C})$ where $a$ is a constant.
\end{lemma}
\begin{proof}
Let $H = \partial^k \bar{\partial}^{k} + c$ and $H_\varphi^* = (\partial^k \bar{\partial}^{k} + c)_\varphi^*$.\\
We suppose that there exists $u \in L^2(\mathbb{\C}, \mathrm{\textnormal{e}}^{- \varphi})$ such that $$ \partial^k \bar{\partial}^{k} u + cu = f$$ with the norm estimate  
$$\vert \vert u \vert \vert_\varphi^2 \leq a,$$ 
the we have
\begin{eqnarray*}
\vert <f, \phi>_\varphi \vert^2 = \vert <Hu, \phi>_\varphi \vert^2 & = & \vert <u, H_\varphi^* \phi>_\varphi \vert^2\\
& \leq & \vert \vert u \vert \vert_\varphi^2 \vert \vert H_\varphi^* \phi \vert \vert_\varphi^2 \\
& \leq & a \vert \vert H_\varphi^* \phi \vert \vert_\varphi^2 \\
\end{eqnarray*}
therefore $$\vert <f, \phi>_\varphi \vert^2 \leq a \vert \vert (\partial^k \bar{\partial}^{k} + c)_\varphi^* \phi \vert \vert_\varphi ^2.$$
We suppose that $$\vert <f, \phi>_\varphi \vert^2 \leq a \vert \vert (\partial^k \bar{\partial}^{k} + c)_\varphi^* \phi \vert \vert_\varphi ^2.$$
Consider the subspace
$$ E = \{ H_\varphi^* \phi : \phi \in C_0^\infty(\mathbb{C}) \} \subset L^2(\mathbb{\C}, \mathrm{\textnormal{e}}^{- \varphi}).$$
We define a linear map by
$$ L_f : E \rightarrow \mathbb{C}$$ with $$L_f(H_\varphi^* \phi) = <f, \phi>_\varphi = \int_\mathbb{C} \overline{f} \phi \mathrm{\textnormal{e}}^{- \varphi} d \sigma$$
however
$$\vert L_f(H_\varphi^* \phi) \vert = \vert <f, \phi>_\varphi \vert \leq \sqrt{a} \vert \vert H_\varphi^* \phi \vert \vert_\varphi$$
therefore $L_f$ is bounded in $E$.\\
According to the Hahn-Banach extension theorem, $L_f$ can be extended to an operator $\tilde{L}_f$ on $L^2(\mathbb{\C}, \mathrm{\textnormal{e}}^{- \varphi})$ with
$$\vert \tilde{L}_f(g)\vert \leq \sqrt{a} \vert \vert g \vert \vert_\varphi \; \; \; \forall \; \; g \in L^2(\mathbb{\C}, \mathrm{\textnormal{e}}^{- \varphi}).$$
According to Riesz's representation theorem, there exists a unique $u_0 \in L^2(\mathbb{\C}, \mathrm{\textnormal{e}}^{- \varphi})$  such that
$$\tilde{L}_f(g) = <u_0, g>_\varphi \; \; \; \forall \; \; g \in L^2(\mathbb{\C}, \mathrm{\textnormal{e}}^{- \varphi})$$
therefore
$$\tilde{L}_f(H_\varphi^* \phi) = <u_0, H_\varphi^* \phi>_\varphi = <Hu_0,  \phi>_\varphi$$
however 
$$\tilde{L}_f(H_\varphi^* \phi) = L_f(H_\varphi^* \phi) =  <f,  \phi>_\varphi$$ in $E$ therefore
$$<Hu_0,  \phi>_\varphi = <f,  \phi>_\varphi$$
so $$ Hu_0 = f$$
which implies
$$\partial^k \bar{\partial}^{k} u_0 + cu_0 = f$$
and
$$\vert \vert u_0 \vert \vert_\varphi ^2 = \vert  <u_0 ,  u_0 >_\varphi \vert = \vert \tilde{L}_f(u_0)\vert \leq \sqrt{a} \vert \vert u_0 \vert \vert_\varphi$$
therefore
$$ \vert \vert u_0 \vert \vert_\varphi ^2 \leq a$$
however $u_0 \in  L^2(\mathbb{\C}, \mathrm{\textnormal{e}}^{- \varphi})$
then let $u_0 = u$, so there exists a $u \in L^2(\mathbb{\C}, \mathrm{\textnormal{e}}^{- \varphi})$ such that
$$\partial^k \bar{\partial}^{k} u + cu = f$$
with
$$\vert \vert u \vert \vert_\varphi ^2 \leq a.$$
\end{proof}
\begin{lemma} \label{B}
$$\vert \vert (\partial^k \bar{\partial}^{k} + c)_\varphi^* \phi \vert \vert_\varphi^2 = \vert \vert (\partial^k \bar{\partial}^{k} + c) \phi \vert \vert^2_\varphi + <\phi,\partial^k \bar{\partial}^{k}( \bar{\partial}^{k*} \partial^{k*} \phi) - \bar{\partial}^{k*} \partial^{k*}(\partial^k \bar{\partial}^{k} \phi) >_\varphi$$
$\forall$ $\phi \in C_0^\infty(\mathbb{C})$.
\end{lemma}
\begin{proof}
Let $H = \partial^k \bar{\partial}^{k} + c$ and $H_\varphi^* = (\partial^k \bar{\partial}^{k} + c)_\varphi^*$.\\
For any $\phi \in C_0^\infty(\mathbb{C})$
\begin{eqnarray*}
\vert \vert H_\varphi^* \phi \vert \vert^2_\varphi & = &  <H_\varphi^* \phi, H_\varphi^* \phi>_\varphi \\
& = &  < \phi, H H_\varphi^* \phi>_\varphi \\ 
& = &  < \phi,  H_\varphi^* H  \phi + H H_\varphi^* \phi - H_\varphi^* H  \phi>_\varphi\\
& = &  <H \phi,  H  \phi> + <\phi, H H_\varphi^* \phi - H_\varphi^* H  \phi>_\varphi \\
& = &   \vert \vert H \phi \vert \vert^2_\varphi + <\phi, H H_\varphi^* \phi - H_\varphi^* H  \phi>_\varphi \\
\end{eqnarray*}
we have\\
\begin{eqnarray*}
H H_\varphi^* \phi & = & (\partial^k \bar{\partial}^{k} + c)(\partial^k \bar{\partial}^{k} + c)_\varphi^* \phi  \\
& = &  (\partial^k \bar{\partial}^{k} + c)( \bar{\partial}_\varphi^{k*} \partial_\varphi^{k*} + c) \phi  \\ 
& = & \partial^k \bar{\partial}^{k}( \bar{\partial}_\varphi^{k*} \partial_\varphi^{k*} \phi) +c \partial^k \bar{\partial}^{k} \phi + c\bar{\partial}_\varphi^{k*} \partial_\varphi^{k*} \phi + c^2 \phi  \\
\end{eqnarray*}
We have the same, 
$$H_\varphi^* H  \phi = \bar{\partial}_\varphi^{k*} \partial_\varphi^{k*}(  \partial^k \bar{\partial}^{k} \phi) +c \partial^k \bar{\partial}^{k} \phi + c\bar{\partial}_\varphi^{k*} \partial_\varphi^{k*} \phi + c^2 \phi$$
therefore
$$H H_\varphi^* \phi - H_\varphi^* H  \phi = \partial^k \bar{\partial}^{k}( \bar{\partial}_\varphi^{k*} \partial_\varphi^{k*} \phi) - \bar{\partial}_\varphi^{k*} \partial_\varphi^{k*}(  \partial^k \bar{\partial}^{k} \phi)$$
so
$$\vert \vert H_\varphi^* \phi \vert \vert^2_\varphi  = \vert \vert H \phi \vert \vert^2_\varphi + <\phi,  \partial^k \bar{\partial}^{k}( \bar{\partial}_\varphi^{k*} \partial_\varphi^{k*} \phi) - \bar{\partial}_\varphi^{k*} \partial_\varphi^{k*}(  \partial^k \bar{\partial}^{k} \phi)>_\varphi$$
therefore
$$\vert \vert (\partial^k \bar{\partial}^{k} + c)_\varphi^* \phi \vert \vert^2_\varphi  = \vert \vert (\partial^k \bar{\partial}^{k} + c) \phi \vert \vert^2_\varphi + <\phi,  \partial^k \bar{\partial}^{k}( \bar{\partial}_\varphi^{k*} \partial_\varphi^{k*} \phi) - \bar{\partial}_\varphi^{k*} \partial_\varphi^{k*}(  \partial^k \bar{\partial}^{k} \phi)>_\varphi.$$
\end{proof}
\begin{lemma} \label{C}
\begin{eqnarray*}
 <\phi,  \partial^k \bar{\partial}^{k}( \bar{\partial}_\varphi^{k*} \partial_\varphi^{k*} \phi) - \bar{\partial}_\varphi^{k*} \partial_\varphi^{k*}(  \partial^k \bar{\partial}^{k} \phi)>_\varphi & = & <\phi, \sum_{i,j=1}^k \sum_{l,m=1}^k C_k^i  C_k^j C_k^l C_k^m \times \\ & & (\partial^{k-m}\bar{\partial}^{k-l}\partial^{k-j}\bar{\partial}^{k-i}\phi) \partial^m \bar{\partial}^l (\mathrm{\textnormal{e}}^{ \varphi} \partial^j \bar{\partial}^i \mathrm{\textnormal{e}}^{- \varphi})>
  \end{eqnarray*}
For any $\phi \in C_0^\infty(\mathbb{C})$.
\end{lemma}
\begin{proof}
\begin{eqnarray*}
 \bar{\partial}_\varphi^{k*} \partial_\varphi^{k*}( \phi) & = & \mathrm{\textnormal{e}}^{ \varphi} \partial^{k} \bar{\partial}^k ( \phi \mathrm{\textnormal{e}}^{- \varphi})\\
 & = & \mathrm{\textnormal{e}}^{ \varphi} \partial^k (\sum_{i=0}^k C_k^i \bar{\partial}^{k-i}\phi \bar{\partial}^i \mathrm{\textnormal{e}}^{- \varphi})\\ 
 & = & \mathrm{\textnormal{e}}^{ \varphi} \sum_{i=0}^k C_k^i \partial^k (\bar{\partial}^{k-i}\phi \bar{\partial}^i \mathrm{\textnormal{e}}^{- \varphi})\\
 & = & \mathrm{\textnormal{e}}^{ \varphi} \sum_{i=0}^k C_k^i \sum_{j=0}^k C_k^j \partial^{k-j}\bar{\partial}^{k-i}\phi \partial^j \bar{\partial}^i \mathrm{\textnormal{e}}^{- \varphi}   
 \\
 & = &  \sum_{i,j=0}^k C_k^i  C_k^j \partial^{k-j} \bar{\partial}^{k-i}\phi (\mathrm{\textnormal{e}}^{ \varphi} \partial^j \bar{\partial}^i \mathrm{\textnormal{e}}^{- \varphi})  
 \\
 \end{eqnarray*}
 \begin{eqnarray*}
  \partial^k \bar{\partial}^{k}( \bar{\partial}_\varphi^{k*} \partial_\varphi^{k*} \phi) & = &  \sum_{i,j=0}^k C_k^i  C_k^j \partial^k \bar{\partial}^{k}[ \partial^{k-j} \bar{\partial}^{k-i}\phi (\mathrm{\textnormal{e}}^{ \varphi} \partial^j \bar{\partial}^i \mathrm{\textnormal{e}}^{- \varphi}) ]\\
  & = & \sum_{i,j=0}^k C_k^i  C_k^j \partial^k \sum_{l=0}^k C_k^l(\bar{\partial}^{k-l}\partial^{k-j} \bar{\partial}^{k-i}\phi) \bar{\partial}^l (\mathrm{\textnormal{e}}^{ \varphi} \partial^j \bar{\partial}^i \mathrm{\textnormal{e}}^{- \varphi})\\
  & = & \sum_{i,j,l=0}^k C_k^i  C_k^j C_k^l \partial^k [(\bar{\partial}^{k-l} \partial^{k-j}\bar{\partial}^{k-i}\phi) \bar{\partial}^l (\mathrm{\textnormal{e}}^{ \varphi} \partial^j \bar{\partial}^i \mathrm{\textnormal{e}}^{- \varphi})]\\
  & = & \sum_{i,j,l=0}^k C_k^i  C_k^j C_k^l \sum_{m=0}^k C_k^m  (\partial^{k-m}\bar{\partial}^{k-l}\partial^{k-j}\bar{\partial}^{k-i}\phi) \partial^m \bar{\partial}^l (\mathrm{\textnormal{e}}^{ \varphi} \partial^j \bar{\partial}^i \mathrm{\textnormal{e}}^{- \varphi})\\
  & = & \sum_{i,j,l,m=0}^k C_k^i  C_k^j C_k^l C_k^m  (\partial^{k-m}\bar{\partial}^{k-l}\partial^{k-j}\bar{\partial}^{k-i}\phi) \partial^m \bar{\partial}^l (\mathrm{\textnormal{e}}^{ \varphi} \partial^j \bar{\partial}^i \mathrm{\textnormal{e}}^{- \varphi}).\\
  \end{eqnarray*}
  We have the same
  $$\bar{\partial}_\varphi^{k*} \partial_\varphi^{k*}(\partial^k \bar{\partial}^{k} \phi)= \sum_{i,j=0}^k C_k^i  C_k^j  (\partial^{k-j}\bar{\partial}^{k-i}\partial^{k}\bar{\partial}^{k}\phi) (\mathrm{\textnormal{e}}^{ \varphi} \partial^j \bar{\partial}^i \mathrm{\textnormal{e}}^{- \varphi}).$$
  then \\
  \begin{eqnarray*}
  \partial^k \bar{\partial}^{k}( \bar{\partial}_\varphi^{k*} \partial_\varphi^{k*} \phi) - \bar{\partial}_\varphi^{k*} \partial_\varphi^{k*}(  \partial^k \bar{\partial}^{k} \phi) & = &  \sum_{i,j=1}^k \sum_{l,m=1}^k C_k^i  C_k^j C_k^l C_k^m \times \\ & & (\partial^{k-m}\bar{\partial}^{k-l}\partial^{k-j} \bar{\partial}^{k-i}\phi) \partial^m \bar{\partial}^l (\mathrm{\textnormal{e}}^{ \varphi} \partial^j \bar{\partial}^i \mathrm{\textnormal{e}}^{- \varphi})
  \end{eqnarray*}
 therefore
  \begin{eqnarray*}
 <\phi, \partial^k \bar{\partial}^{k}( \bar{\partial}_\varphi^{k*} \partial_\varphi^{k*} \phi) - \bar{\partial}_\varphi^{k*} \partial_\varphi^{k*}(  \partial^k \bar{\partial}^{k} \phi)> & = & <\phi, \sum_{i,j=1}^k \sum_{l,m=1}^k C_k^i  C_k^j C_k^l C_k^m \times \\ & & (\partial^{k-m}\bar{\partial}^{k-l}\partial^{k-j}\bar{\partial}^{k-i}\phi) \partial^m \bar{\partial}^l (\mathrm{\textnormal{e}}^{ \varphi} \partial^j \bar{\partial}^i \mathrm{\textnormal{e}}^{- \varphi})>.
  \end{eqnarray*}
  \end{proof}
  \begin{lemma} \label{B}
  Let $\varphi = \vert z \vert ^2$ then
\begin{eqnarray*}
 <\phi,  \partial^k \bar{\partial}^{k}( \bar{\partial}_\varphi^{k*} \partial_\varphi^{k*} \phi) - \bar{\partial}_\varphi^{k*} \partial_\varphi^{k*}(  \partial^k \bar{\partial}^{k} \phi)>_\varphi & = & <\phi, \sum_{j=1}^k \sum_{l=1}^k \sum_{i=j-n}^k  \sum_{n=l}^j\sum_{m=1}^{i-j+n} C_k^i  C_k^j C_k^l C_k^m C_j^n\times \\ & & (\partial^{k-m}\bar{\partial}^{k-l}\partial^{k-j}\bar{\partial}^{k-i}\phi)\times \\ & & (-1)^{i+n}  \frac{i!}{(i-j+n-m)!} \\ & &  \frac{(n)!}{(n-l)!} (z)^{i-j+n-m} (\bar{z})^{n-l}>
 \end{eqnarray*}
For all $\phi \in C_0^\infty(\mathbb{C})$.
\end{lemma}
\begin{proof}
 Let $h(g) = \mathrm{\textnormal{e}}^{g}$ with $g = - \varphi$, according to the formula of Faà di Bruno \cite{3}, we have
 \begin{eqnarray*}
\bar{\partial}^i \mathrm{\textnormal{e}}^{g} & = & \bar{\partial}^i(h(g))\\
 & = &  \sum \frac{i!}{m_1!m_2! \cdots m_i!}h^{(m_1 + \cdots + m_i)}(g)\prod_{\gamma=1}^i (\frac{\bar{\partial}^\gamma g}{\gamma !})^{m_\gamma}\\ 
 & = &  (\sum \frac{i!}{m_1!m_2! \cdots m_i!} \prod_{\gamma=1}^i (\frac{-\bar{\partial}^\gamma \varphi}{\gamma !})^{m_\gamma})\mathrm{\textnormal{e}}^{- \varphi}\\
 & =: & P_i \mathrm{\textnormal{e}}^{- \varphi}\\
 \end{eqnarray*}
 Where the sum runs through all the $i$-tuples of nonnegative integers $(m_1, m_2, \cdots , m_i)$ satisfying the constraint : $1m_1 + 2m_2+ \cdots + im_i = i$.\\
 Let $\varphi = \vert z \vert ^2$, we have
 \[ 
 \bar{\partial}^\gamma \varphi = \left \{
 \begin{array}{rl}
 z & \mbox{ if } \gamma = 1\\
 0 & \mbox{ if } \gamma \geq 2
 \end{array}
 \right.
 \]
 $$P_i = (-\bar{\partial}^\gamma \varphi)^i = (-z)^i = (-1)^i (z)^i$$
 $$ \bar{\partial}^i \mathrm{\textnormal{e}}^{-\vert z \vert ^2} =  (-1)^i (z)^i \mathrm{\textnormal{e}}^{-\vert z \vert ^2}$$
 So
\begin{equation} \label{D}
  \partial^j \bar{\partial}^i \mathrm{\textnormal{e}}^{-\vert z \vert ^2} = (-1)^i \partial^j(z)^i \mathrm{\textnormal{e}}^{-\vert z \vert ^2})
  \end{equation}
  \begin{equation} \label{D}
  \partial^j \bar{\partial}^i \mathrm{\textnormal{e}}^{-\vert z \vert ^2} = (-1)^i \sum_{n=0}^j  C_j^n  \partial^{j-n}(z)^i)\partial^n \mathrm{\textnormal{e}}^{-\vert z \vert ^2}
  \end{equation}
By still applying Faà di Bruno's formula, we have
 $$P_j  = (-1)^j (z)^j$$
 $$ \partial^n \mathrm{\textnormal{e}}^{-\vert z \vert ^2} =  (-1)^n (\bar{z})^n \mathrm{\textnormal{e}}^{-\vert z \vert ^2}$$
 By replacing in \ref{D}, we have
 \begin{equation} \label{E}
 \partial^j \bar{\partial}^i \mathrm{\textnormal{e}}^{-\vert z \vert ^2} = (-1)^i \sum_{n=0}^j (-1)^n C_j^n  \partial^{j-n}(z)^i (\bar{z})^n \mathrm{\textnormal{e}}^{-\vert z \vert ^2}
 \end{equation}
 however
 \[ 
 \partial^{j-n} (z)^i = \left \{
 \begin{array}{lr}
 0 & \mbox{ if } i < j-n \\ & \\
 \frac{i!}{(i-j+n)!} (z)^{i-j+n} & \mbox{ if } i \geq j-n
 \end{array}
 \right.
 \]
 By replacing in \ref{E}, we have
 \[ 
 \partial^j \bar{\partial}^i \mathrm{\textnormal{e}}^{-\vert z \vert ^2} = \left \{
 \begin{array}{lr}
 0 & \mbox{ if } i < j-n \\ & \\
 \sum_{n=0}^j (-1)^{n+i} C_j^n  \frac{i!}{(i-j+n)!}  (z)^{i-j+n} (\bar{z})^n \mathrm{\textnormal{e}}^{-\vert z \vert ^2} & \mbox{ if } i \geq j-n
 \end{array}
 \right.
 \]
 \;\\
 
  \[ 
 \mathrm{\textnormal{e}}^{\vert z \vert ^2}\partial^j \bar{\partial}^i \mathrm{\textnormal{e}}^{-\vert z \vert ^2} = \left \{
 \begin{array}{lr}
 0 & \mbox{ if } i < j-n \\ & \\
 \sum_{n=0}^j (-1)^{n+i} C_j^n  \frac{i!}{(i-j+n)!}  (z)^{i-j+n} (\bar{z})^n  & \mbox{ if } i \geq j-n
 \end{array}
 \right.
 \]
we have
$$\bar{\partial}^l(\mathrm{\textnormal{e}}^{\vert z \vert ^2} \partial^j \bar{\partial}^i \mathrm{\textnormal{e}}^{-\vert z \vert ^2}) =  \sum_{n=0}^j (-1)^{n+i} C_j^n  \frac{i!}{(i-j+n)!}  (z)^{i-j+n} \bar{\partial}^l(\bar{z})^n   \mbox{ if } i \geq j-n $$
\;\\
\[ 
\bar{\partial}^l( \mathrm{\textnormal{e}}^{\vert z \vert ^2} \partial^j \bar{\partial}^i \mathrm{\textnormal{e}}^{-\vert z \vert ^2}) = \left \{
 \begin{array}{lr}
 0  & \\ \mbox{ if } n < l \mbox{ and } i < j-n & \\& \\& \\
 \sum_{n=0}^j (-1)^{n+i} C_j^n  \frac{i!}{(i-j+n)!} \frac{(n)!}{(n-l)!} (z)^{i-j+n} (\bar{z})^{n-l}  & \\ & \\ \mbox{ if } i \geq j-n \mbox{ and } n \geq l
 \end{array}
 \right.
 \]
 so
\begin{eqnarray*} 
\partial^m \bar{\partial}^l( \mathrm{\textnormal{e}}^{\vert z \vert ^2} \partial^j \bar{\partial}^i \mathrm{\textnormal{e}}^{-\vert z \vert ^2}) = \left \{
 \begin{array}{lr}
 0  & \\ \mbox{ if } n < l, \; i-j+n < m \mbox{ and } i < j-n & \\& \\& \\
 \sum_{n=0}^j (-1)^{n+i} C_j^n  \frac{i!}{(i-j+n)!} \frac{(n)!}{(n-l)!} \frac{(i-j+n)!}{(i-j+n-m)!}   (z)^{i-j+n-m} (\bar{z})^{n-l}  & \\ & \\ \mbox{ if } i \geq j-n, \;  i-j+n \geq m \mbox{ and } n \geq l &
 \end{array}
 \right.
 \end{eqnarray*}
\begin{align*}
<\phi, \partial^k \bar{\partial}^{k}( \bar{\partial}_\varphi^{k*} \partial_\varphi^{k*} \phi) - \bar{\partial}_\varphi^{k*} \partial_\varphi^{k*}(  \partial^k \bar{\partial}^{k} \phi)>  =<\phi, &\\ \sum_{j=1}^k \sum_{l=1}^k \sum_{i=j-n}^k  \sum_{n=l}^j\sum_{m=1}^{i-j+n} C_k^i  C_k^j C_k^l C_k^m C_j^n  (\partial^{k-m}\bar{\partial}^{k-l}\partial^{k-j}\bar{\partial}^{k-i}\phi) (-1)^{i+n} \times &\\ \frac{i!}{(i-j+n-m)!} \frac{(n)!}{(n-l)!} (z)^{i-j+n-m} (\bar{z})^{n-l}>.&\\
\end{align*}
 \end{proof}
 \begin{lemma}
  Let $\varphi = \vert z \vert ^2$ then
\begin{eqnarray*}
<\phi, \partial^k \bar{\partial}^{k}( \bar{\partial}_\varphi^{k*} \partial_\varphi^{k*} \phi) - \bar{\partial}_\varphi^{k*} \partial_\varphi^{k*}(  \partial^k \bar{\partial}^{k} \phi)>_\varphi & = & \sum_{\alpha =0}^{k-1} \sum_{\beta =k-i+j-n}^{k-1}\frac{(k!)^4}{(\alpha!)^2(\beta!)^2((k-\beta)!)((k-\alpha)!)} \times\\ & &  \vert \vert \bar{\partial}^{\beta} \partial^{\alpha} \phi \vert \vert_\varphi^2
\end{eqnarray*}

 \end{lemma}
 \begin{proof}
 According to the lemma \ref{B}, we have
 \begin{align*}
<\phi, \partial^k \bar{\partial}^{k}( \bar{\partial}_\varphi^{k*} \partial_\varphi^{k*} \phi) - \bar{\partial}_\varphi^{k*} \partial_\varphi^{k*}(  \partial^k \bar{\partial}^{k} \phi)>  =<\phi, &\\ \sum_{j=1}^k \sum_{l=1}^k \sum_{i=j-n}^k  \sum_{n=l}^j\sum_{m=1}^{i-j+n} C_k^i  C_k^j C_k^l C_k^m C_j^n  (\partial^{k-m}\bar{\partial}^{k-l}\partial^{k-j}\bar{\partial}^{k-i}\phi) (-1)^{i+n} \times &\\ \frac{i!}{(i-j+n-m)!} \frac{(n)!}{(n-l)!} (z)^{i-j+n-m} (\bar{z})^{n-l}>&\\
\end{align*}
Let $s = n-l$
\begin{align*}
<\phi, \partial^k \bar{\partial}^{k}( \bar{\partial}_\varphi^{k*} \partial_\varphi^{k*} \phi) - \bar{\partial}_\varphi^{k*} \partial_\varphi^{k*}(  \partial^k \bar{\partial}^{k} \phi)>  =<\phi, &\\ \sum_{j=1}^k \sum_{l=1}^k \sum_{i=j-n}^k  \sum_{m=1}^{i-j+n} A (\partial^{k-m}\bar{\partial}^{k-l}\partial^{k-j}\bar{\partial}^{k-i}\phi) \sum_{s=0}^{j-l}    (-1)^{i+l-m+m+s}C_{j-l}^s \times &\\ \frac{(i-m)!}{(i-j+l+s-m)!} (z)^{i-j-m+l+s} (\bar{z})^{s}>&\\
\end{align*}
with
 $$A=A(k,j,l,m,i) := \frac{(k!)^4}{m!l!(k-l)!(k-m)! (k-j)!(k-i)!(j-l)!(i-m)!}$$
\begin{align*}
<\phi, \partial^k \bar{\partial}^{k}( \bar{\partial}_\varphi^{k*} \partial_\varphi^{k*} \phi) - \bar{\partial}_\varphi^{k*} \partial_\varphi^{k*}(  \partial^k \bar{\partial}^{k} \phi)>  =<\phi, &\\ \sum_{j=1}^k \sum_{l=1}^k \sum_{i=j-n}^k  \sum_{m=1}^{i-j+n} A (\partial^{k-m}\bar{\partial}^{k-l}\partial^{k-j}\bar{\partial}^{k-i}\phi)(-1)^{m+l}  \mathrm{\textnormal{e}}^{\vert z \vert ^2}\partial^{j-l} \bar{\partial}^{i-m} \mathrm{\textnormal{e}}^{-\vert z \vert ^2}>&\\
\end{align*}

\begin{align*}
[ \partial^k \bar{\partial}^{k}( \bar{\partial}_\varphi^{k*} \partial_\varphi^{k*} \phi) - \bar{\partial}_\varphi^{k*} \partial_\varphi^{k*}(  \partial^k \bar{\partial}^{k} \phi)] \mathrm{\textnormal{e}}^{-\vert z \vert ^2} = &\\ \sum_{j=1}^k \sum_{l=1}^k \sum_{i=j-n}^k  \sum_{m=1}^{i-j+n} A (\partial^{k-m}\bar{\partial}^{k-l}\partial^{k-j}\bar{\partial}^{k-i}\phi)(-1)^{m+l} \partial^{j-l} \bar{\partial}^{i-m} \mathrm{\textnormal{e}}^{-\vert z \vert ^2}&\\
\end{align*} 
\begin{remark}
Since $i \geq j-n, \;  i-j+n \geq m \mbox{ and } n \geq l$ then $j \geq l$ and $i \geq m$. If $j < l$ or $i < m$ then $$ \partial^k \bar{\partial}^{k}( \bar{\partial}_\varphi^{k*} \partial_\varphi^{k*} \phi) - \bar{\partial}_\varphi^{k*} \partial_\varphi^{k*}(  \partial^k \bar{\partial}^{k} \phi)] \mathrm{\textnormal{e}}^{-\vert z \vert ^2} = 0.$$
\end{remark}

 Let $s = j-l$ and $t = i-m$ then
 \begin{align*}
[ \partial^k \bar{\partial}^{k}( \bar{\partial}_\varphi^{k*} \partial_\varphi^{k*} \phi) - \bar{\partial}_\varphi^{k*} \partial_\varphi^{k*}(  \partial^k \bar{\partial}^{k} \phi)] \mathrm{\textnormal{e}}^{-\vert z \vert ^2} = &\\ \sum_{l=1}^k \sum_{m=1}^{i-j+n} A' \sum_{s=0}^{k-l} \sum_{t=0}^{k-m}C_{k-m}^s  C_{k-l}^t (\partial^{k-m}\bar{\partial}^{k-l}\partial^{k-l-s}\bar{\partial}^{k-m-t}\phi)(-1)^{m+l} \partial^{s} \bar{\partial}^{t} \mathrm{\textnormal{e}}^{-\vert z \vert ^2}&\\
\end{align*} 
with
 $$A'=A'(k,l,m) := \frac{(k!)^4}{m!l![(k-l)!]^2[(k-m)!]^2}$$
\begin{align*}
[ \partial^k \bar{\partial}^{k}( \bar{\partial}_\varphi^{k*} \partial_\varphi^{k*} \phi) - \bar{\partial}_\varphi^{k*} \partial_\varphi^{k*}(  \partial^k \bar{\partial}^{k} \phi)] \mathrm{\textnormal{e}}^{-\vert z \vert ^2} = &\\ \sum_{l=1}^k \sum_{m=1}^{i-j+n} A'(-1)^{m+l} \partial^{k-m} \bar{\partial}^{k-l}(\partial^{k-l}\bar{\partial}^{k-m}\phi \mathrm{\textnormal{e}}^{-\vert z \vert ^2})&\\
\end{align*} 

  By returning to the computation with respect of weighted inner product we have,
  \begin{eqnarray*}
  <\phi, \partial^k \bar{\partial}^{k}( \bar{\partial}_\varphi^{k*} \partial_\varphi^{k*} \phi) - \bar{\partial}_\varphi^{k*} \partial_\varphi^{k*}(  \partial^k \bar{\partial}^{k} \phi)>_\varphi & = & \int_\mathbb{C} \overline{\phi}\partial^k \bar{\partial}^{k}( \bar{\partial}_\varphi^{k*} \partial_\varphi^{k*} \phi) \mathrm{\textnormal{e}}^{-\vert z \vert ^2} d\sigma - \\ & & \int_\mathbb{C} \overline{\phi}\bar{\partial}_\varphi^{k*} \partial_\varphi^{k*}(  \partial^k \bar{\partial}^{k} \phi)\mathrm{\textnormal{e}}^{-\vert z \vert ^2} d\sigma\\
  & = & \sum_{m=1}^k \sum_{l=1}^{i-j+n} A'(-1)^{m+l} \times\\ & & \int_\mathbb{C} \overline{\phi} \partial^{k-m}\bar{\partial}^{k-l}[(\partial^{k-l}\bar{\partial}^{k-m}\phi)  \mathrm{\textnormal{e}}^{-\vert z \vert ^2}] d\sigma\\
  & = & \sum_{m=1}^k \sum_{l=1}^{i-j+n} A'(-1)^{m+l} \times \\ & &(-1)^{-m-l} \int_\mathbb{C} \bar{\partial}^{k-l} \partial^{k-m}\overline{\phi} (\partial^{k-l}\bar{\partial}^{k-m}\phi)  \mathrm{\textnormal{e}}^{-\vert z \vert ^2} d\sigma\\
  & = & \sum_{m=1}^k \sum_{l=1}^{i-j+n} A' \times\\ & & \int_\mathbb{C} \overline{\partial^{k-l}\bar{\partial}^{k-m} \phi} (\partial^{k-l}\bar{\partial}^{k-m}\phi)  \mathrm{\textnormal{e}}^{-\vert z \vert ^2} d\sigma\\
  & = & \sum_{m=1}^k \sum_{l=1}^{i-j+n} A' \times\\ & & <\partial^{k-l}\bar{\partial}^{k-m} \phi, \partial^{k-l} \bar{\partial}^{k-m} \phi>_\varphi\\
  & = & \sum_{m=1}^k \sum_{l=1}^{i-j+n} A' \vert \vert \partial^{k-l} \bar{\partial}^{k-m} \phi \vert \vert_\varphi^2\\
  & = & \sum_{m=1}^k \sum_{l=1}^{i-j+n} \frac{ (k!)^4}{m!l!((k-l)!)^2((k-m)!)^2}\times\\ & &\vert \vert \partial^{k-l} \bar{\partial}^{k-m} \phi \vert \vert_\varphi^2\\
  \end{eqnarray*}
  Let $\alpha = k-l$ and $\beta = k-m$
$$<\phi, \partial^k \bar{\partial}^{k}( \bar{\partial}_\varphi^{k*} \partial_\varphi^{k*} \phi) - \bar{\partial}_\varphi^{k*} \partial_\varphi^{k*}(  \partial^k \bar{\partial}^{k} \phi)>_\varphi = \sum_{\alpha=0}^{k-1} \sum_{\beta=0}^{k-1} \frac{(k!)^4}{(\alpha!)^2(\beta!)^2((k-\beta)!)((k-\alpha)!)}\vert \vert \partial^{\alpha} \bar{\partial}^{\beta} \phi \vert \vert_\varphi^2.$$
\end{proof}
\section{Proof of the main Theorem}
\begin{proof}{(Theorem \ref{P})}
Let $\varphi = \vert z \vert^2$ and $\phi \in C_0^\infty(\mathbb{C})$. According to the lemma \ref{B}
$$\vert \vert (\partial^k \bar{\partial}^{k} + c)_\varphi^* \phi \vert \vert_\varphi^2 \geq <\phi,\partial^k \bar{\partial}^{k}( \bar{\partial}^{k*} \partial^{k*} \phi) - \bar{\partial}^{k*} \partial^{k*}(\partial^k \bar{\partial}^{k} \phi) >_\varphi.$$
According to the lemma \ref{C}
\begin{eqnarray*}
 <\phi,\partial^k \bar{\partial}^{k}( \bar{\partial}^{k*} \partial^{k*} \phi) - \bar{\partial}^{k*} \partial^{k*}(\partial^k \bar{\partial}^{k} \phi) >_\varphi & = & \sum_{\alpha=0}^{k-1} \sum_{\beta=0}^{k-j} \frac{(k!)^4}{(\alpha!)^2(\beta!)^2((k-\beta)!)((k-\alpha)!)}  \\& & \vert \vert \partial^{\alpha} \bar{\partial}^{\beta} \phi \vert \vert_\varphi^2\\
 \end{eqnarray*}
 therefore
 $$\vert \vert (\partial^k \bar{\partial}^{k} + c)_\varphi^* \phi \vert \vert_\varphi^2 \geq (k!)^2 \vert \vert \phi \vert \vert_\varphi^2.$$
 According to the Cauchy-Schwarz inegality, we have
 \begin{eqnarray*}
\vert<f, \phi>_\varphi \vert^2 & \leq & \; \vert \vert f \vert \vert_\varphi^2 \vert \vert \phi \vert \vert_\varphi^2\\
& = &(\frac{1}{(k!)^2} \vert \vert f \vert \vert_\varphi^2) ((k!)^2 \vert \vert \phi \vert \vert_\varphi^2)\\
& \leq & (\frac{1}{(k!)^2} \vert \vert f \vert \vert_\varphi^2) \vert \vert (\partial^k \bar{\partial}^{k} + c)_\varphi^* \phi \vert \vert_\varphi^2.\\ 
 \end{eqnarray*}
 According to the lemma \ref{A}, there exists a $u \in  L^2(\mathbb{\C}, \mathrm{\textnormal{e}}^{- \vert z \vert^2})$ such that
 $$ \partial^k \bar{\partial}^{k} u + cu = f$$ with the norm estimate  
 $$\vert \vert u \vert \vert_\varphi^2 \leq \frac{1}{(k!)^2} \vert \vert f \vert \vert_\varphi^2$$
 that is to say
 $$\int_\mathbb{C} \vert u \vert^2 \mathrm{\textnormal{e}}^{- \vert z \vert^2} d\sigma \leq \frac{1}{(k!)^2} \int_\mathbb{C} \vert f \vert^2 \mathrm{\textnormal{e}}^{- \vert z \vert^2} d\sigma.$$
 Moreover for $k=1$, for all $\phi \in C_0^\infty(\mathbb{C})$, we have
$$\partial^1 \bar{\partial}^1( \bar{\partial}_\varphi^{1*} \partial_\varphi^{1*} \phi) - \bar{\partial}_\varphi^{1*} \partial_\varphi^{1*}(  \partial^1 \bar{\partial}^1 \phi) = \phi \partial^1 \bar{\partial}^1(\mathrm{\textnormal{e}}^{\varphi} \partial^1 \bar{\partial}^1 \mathrm{\textnormal{e}}^{-  \varphi }).$$
According to the lemma \ref{B}, we have
\begin{eqnarray*}
\vert \vert (\partial^1 \bar{\partial}^1 + c)_\varphi^* \phi \vert \vert_\varphi^2 & \geq & <\phi, \partial^1 \bar{\partial}^1( \bar{\partial}_\varphi^{1*} \partial_\varphi^{1*} \phi) - \bar{\partial}_\varphi^{1*} \partial_\varphi^{1*}(  \partial^1 \bar{\partial}^1 \phi) >_\varphi\\
& = & <\phi, \phi \partial^1 \bar{\partial}^1(\mathrm{\textnormal{e}}^{\varphi} \partial^1 \bar{\partial}^1 \mathrm{\textnormal{e}}^{-  \varphi })>_\varphi\\
& = & <\phi\sqrt{ \partial^1 \bar{\partial}^1(\mathrm{\textnormal{e}}^{\varphi} \partial^1 \bar{\partial}^1 \mathrm{\textnormal{e}}^{-  \varphi })}, \phi \sqrt{ \partial^1 \bar{\partial}^1(\mathrm{\textnormal{e}}^{\varphi} \partial^1 \bar{\partial}^1 \mathrm{\textnormal{e}}^{-  \varphi })}>_\varphi\\
& = &  \Big\Vert \phi\sqrt{\partial^1 \bar{\partial}^1(\mathrm{\textnormal{e}}^{ \varphi} \partial^1 \bar{\partial}^1 \mathrm{\textnormal{e}}^{-  \varphi })}  \Big\Vert_\varphi^2 \\
\end{eqnarray*}
By applying the Cauchy-Schwarz inequality, we have
\begin{eqnarray*}
\vert <f, \phi >_\varphi \vert ^2 & = & \Big\vert <\frac{f}{\sqrt{\partial^1 \bar{\partial}^1(\mathrm{\textnormal{e}}^{ \varphi } \partial^1 \bar{\partial}^1 \mathrm{\textnormal{e}}^{-  \varphi })}}, \phi \sqrt{\partial^1 \bar{\partial}^1(\mathrm{\textnormal{e}}^{ \varphi } \partial^1 \bar{\partial}^1 \mathrm{\textnormal{e}}^{- \varphi })}>_\varphi \Big\vert^2\\
& \leq &  \Big\Vert \frac{f}{\sqrt{\partial^1 \bar{\partial}^1(\mathrm{\textnormal{e}}^{ \varphi } \partial^1 \bar{\partial}^1 \mathrm{\textnormal{e}}^{-  \varphi })}} \Big\Vert_\varphi^2  \Big\Vert \phi \sqrt{\partial^1 \bar{\partial}^1(\mathrm{\textnormal{e}}^{ \varphi} \partial^1 \bar{\partial}^1 \mathrm{\textnormal{e}}^{- \varphi })}  \Big\Vert_\varphi^2\\
& \leq & \Big\Vert \frac{f}{\sqrt{\partial^1 \bar{\partial}^1(\mathrm{\textnormal{e}}^{\varphi} \partial^1 \bar{\partial}^1 \mathrm{\textnormal{e}}^{- \varphi})}} \Big\Vert_\varphi^2  \Big\Vert (\partial^1 \bar{\partial}^1+ c)_\varphi^* \phi  \Big\Vert_\varphi^2\\
\end{eqnarray*}
According to the lemma \ref{A}, there exists a weak solution $u \in L^2(\mathbb{\C}, \mathrm{\textnormal{e}}^{-\varphi})$ of the equation
$$ \partial^1 \bar{\partial}^1 u + cu = f$$ with
$$\vert \vert u \vert \vert_\varphi^2 \leq   \Big\Vert \frac{f}{\sqrt{\partial^1 \bar{\partial}^1(\mathrm{\textnormal{e}}^{ \varphi} \partial^1 \bar{\partial}^1 \mathrm{\textnormal{e}}^{-  \varphi })}} \Big\Vert_\varphi^2$$
that is to say
$$\int_\mathbb{C} \vert u \vert^2 \mathrm{\textnormal{e}}^{-  \varphi } d\sigma \leq 16 \int_\mathbb{C} \frac{\vert f \vert^2}{\Delta(\mathrm{\textnormal{e}}^{ \varphi} \Delta \mathrm{\textnormal{e}}^{-  \varphi })}  \mathrm{\textnormal{e}}^{-  \varphi } d\sigma.$$
\end{proof}
\begin{theorem}
There is a liner operator and bounded $$T_k : L^2(\mathbb{\C}, \mathrm{\textnormal{e}}^{- \vert z \vert^2}) \longrightarrow L^2(\mathbb{\C}, \mathrm{\textnormal{e}}^{- \vert z \vert^2})$$
such that  $$ (\partial^k \bar{\partial}^{k}  + c) T_k = I$$ with the norm estimate   
$$\vert \vert T_k \vert \vert_\varphi \leq \frac{1}{(k!)^2}$$ where $\vert \vert T_k \vert \vert_\varphi$ is the norm of $T_k$ in $L^2(\mathbb{\C}, \mathrm{\textnormal{e}}^{- \vert z \vert^2})$.
\end{theorem}
\begin{proof}
Let $f \in L^2(\mathbb{\C}, \mathrm{\textnormal{e}}^{- \vert z \vert^2})$. According to the theorem \ref{P}, there is a weak solution $u \in L^2(\mathbb{\C}, \mathrm{\textnormal{e}}^{- \vert z \vert^2})$ such that
$$ \partial^k \bar{\partial}^{k} u + cu = f$$ with the norm estimate  
 $$\vert \vert u \vert \vert_\varphi \leq \frac{1}{(k!)^2} \vert \vert f \vert \vert_\varphi$$
therefore 
$$ (\partial^k \bar{\partial}^{k}  + c)T_k(f) = f$$ with
 $$\vert \vert T_k(f) \vert \vert_\varphi \leq \frac{1}{(k!)} \vert \vert f \vert \vert_\varphi.$$
So
 $$ (\partial^k \bar{\partial}^{k}  + c)T_k = I$$ with 
 $$\vert \vert T_k \vert \vert_\varphi \leq \frac{1}{(k!)^2}.$$
\end{proof}

\section{Some consequences}
Let $\lambda > 0$ and $z_0 \in \mathbb{C}$. Let $\varphi = \lambda^2 \vert z - z_0 \vert^2$, we obtain the corollary of the following Theorem \ref{P} :
\begin{coro} \label{2}
Let $f \in L^2(\mathbb{\C}, \mathrm{\textnormal{e}}^{-\lambda^2 \vert z - z_0 \vert^2})$, there exists a weak solution $u \in L^2(\mathbb{\C}, \mathrm{\textnormal{e}}^{-\lambda^2 \vert z - z_0 \vert^2})$ of the equation
$$ \partial^k \bar{\partial}^{k} u + cu = f$$ with the norm estimate  
$$\int_\mathbb{C} \vert u \vert^2 \mathrm{\textnormal{e}}^{- \lambda^2 \vert z - z_0 \vert^2} d\sigma \leq \frac{1}{(\lambda^k k!)^2} \int_\mathbb{C} \vert f \vert^2 \mathrm{\textnormal{e}}^{- \lambda^2 \vert z - z_0 \vert^2} d\sigma.$$
\end{coro}
\begin{proof}
Let $f \in L^2(\mathbb{\C}, \mathrm{\textnormal{e}}^{-\lambda^2 \vert z - z_0 \vert^2})$, we have $$\int_\mathbb{C} \vert f \vert^2 \mathrm{\textnormal{e}}^{- \lambda^2 \vert z - z_0 \vert^2} d\sigma < + \infty.$$
Lat $z = \frac{w}{\lambda} + z_0$ and $g(w)= f(z)= f(\frac{w}{\lambda} + z_0)$ then
$$\frac{1}{\lambda^2} \int_\mathbb{C} \vert g \vert^2 \mathrm{\textnormal{e}}^{- \vert w \vert^2} d\sigma(w) < + \infty$$
therefore $g \in L^2(\mathbb{\C}, \mathrm{\textnormal{e}}^{-\vert w \vert^2})$. According to the Theorem \ref{P}, there exists a weak solution $v \in L^2(\mathbb{\C}, \mathrm{\textnormal{e}}^{-\vert w \vert^2})$ of the equation
$$ \partial^k \bar{\partial}^{k} v + \frac{c}{\lambda^k} v = g$$ with the norm estimate  
$$\int_\mathbb{C} \vert v \vert^2 \mathrm{\textnormal{e}}^{- \vert w \vert^2} d\sigma(w) \leq \frac{1}{( k!)^2} \int_\mathbb{C} \vert g \vert^2 \mathrm{\textnormal{e}}^{- \vert w \vert^2} d\sigma(w).$$
Let $u(z) = \frac{1}{\lambda^k} v(w) = \frac{1}{\lambda^k} v(\lambda (z - z_0))$.\\
Then
$$ \partial^k \bar{\partial}^{k} u + cu = f$$ with the norm estimate  
$$\int_\mathbb{C} \vert u \vert^2 \mathrm{\textnormal{e}}^{- \lambda^2 \vert z - z_0 \vert^2} d\sigma \leq \frac{1}{(\lambda^k k!)^2} \int_\mathbb{C} \vert f \vert^2 \mathrm{\textnormal{e}}^{- \lambda^2 \vert z - z_0 \vert^2} d\sigma.$$
\end{proof}\\
As a consequence of the Corollary \ref{2} we have
\begin{coro} \label{3}
Let $U \in \mathbb{C}$ be a bounded open set.\\
Let $f \in L^2(U)$, there exists a weak solution $u \in L^2(U)$ of the equation
$$ \partial^k \bar{\partial}^{k} u + cu = f$$ with 
$$\vert \vert u \vert \vert_{L^2(U)} \leq \Big(\frac{{\textnormal{e}}^{\vert U \vert^2}}{(k!)^2}\Big) \vert \vert f \vert \vert_{L^2(U)}$$ 
where $\vert U \vert $ is the diameter of $U$.
\end{coro}
\begin{proof}
Let $z_0 \in U$ and $f \in L^2(U)$, we have
 \[ 
 \tilde{f}(z) = \left \{
 \begin{array}{rl}
 f(z) & \mbox{ si } z \in U \\
 0 & \mbox{ si } z \in \mathbb{C} \setminus U
 \end{array}
 \right.
 \]
 Therefore $\tilde{f} \in L^2(\mathbb{C}) \subset L^2(\mathbb{\C}, \mathrm{\textnormal{e}}^{-\vert z - z_0 \vert^2})$. According to the Corollary \ref{2}, there exists a weak solution $\tilde{u} \in  L^2(\mathbb{\C}, \mathrm{\textnormal{e}}^{-\vert z - z_0 \vert^2})$ of the equation $$ \partial^k \bar{\partial}^{k} \tilde{u} + c\tilde{u} = \tilde{f}$$ with
 $$\int_\mathbb{C} \vert \tilde{u} \vert^2 \mathrm{\textnormal{e}}^{-  \vert z - z_0 \vert^2} d\sigma \leq \frac{1}{(k!)^2} \int_\mathbb{C} \vert \tilde{f} \vert^2 \mathrm{\textnormal{e}}^{-  \vert z - z_0 \vert^2} d\sigma$$
 therefore
 $$\int_\mathbb{C} \vert \tilde{u} \vert^2 \mathrm{\textnormal{e}}^{-  \vert z - z_0 \vert^2} d\sigma \leq \frac{1}{(k!)^2} \int_U \vert f \vert^2  d\sigma$$
 however $$\int_\mathbb{C} \vert \tilde{u} \vert^2 \mathrm{\textnormal{e}}^{-  \vert z - z_0 \vert^2} d\sigma \geq \mathrm{\textnormal{e}}^{-\vert U \vert^2} \int_U \vert \tilde{u} \vert^2  d\sigma$$
 therefore
 $$\int_U \vert \tilde{u} \vert^2  d\sigma \leq \Big(\frac{{\textnormal{e}}^{\vert U \vert^2}}{(k!)^2}\Big) \int_U \vert f \vert^2  d\sigma.$$
 Let $\tilde{u}_{\mid U} = u$ then we have
 $$ \partial^k \bar{\partial}^{k} u + cu = f$$  with
$$\vert \vert u \vert \vert_{L^2(U)} \leq \Big(\frac{{\textnormal{e}}^{\vert U \vert^2}}{(k!)^2}\Big) \vert \vert f \vert \vert_{L^2(U)}.$$ 
\end{proof}
 \section*{Acknowledgments}
 The authors would like to thank the director of  Ziguinchor Polytechnic Institute (ZIP) and all the members of the SAGELAB laboratory at ZIP where we stayed to do  this work.

\end{document}